\documentclass{amsproc}
\usepackage{euscript}
\usepackage{cases}
\usepackage{mathrsfs}
\usepackage{bbm}
\usepackage{amssymb}
\usepackage{amsfonts,amsmath,amsxtra,mathdots,mathabx}
\usepackage{color}
\usepackage{hyperref}
\usepackage{tikz}
\usepackage{appendix,upgreek}
\allowdisplaybreaks

\def\SG{\mathscr{G}} 

\def\SF{\mathscr{F}}

\def\SQ{\mathscr{Q}}

\DeclareFontFamily{U}{matha}{\hyphenchar\font45}
\DeclareFontShape{U}{matha}{m}{n}{
	<5> <6> <7> <8> <9> <10> gen * matha
	<10.95> matha10 <12> <14.4> <17.28> <20.74> <24.88> matha12
}{}
\DeclareSymbolFont{matha}{U}{matha}{m}{n}

\DeclareMathSymbol{\Lt}{3}{matha}{"CE}
\DeclareMathSymbol{\Gt}{3}{matha}{"CF}

\DeclareSymbolFont{mathc}{OML}{txmi}{m}{it}
\DeclareMathSymbol{\varuu}{\mathord}{mathc}{117}
\DeclareMathSymbol{\varvv}{\mathord}{mathc}{118}
\DeclareMathSymbol{\varww}{\mathord}{mathc}{119}

\def\SO {\text{\raisebox{- 2 \depth}{\scalebox{1.1}{$ \text{\usefont{U}{BOONDOX-calo}{m}{n}O}  $}}}}

\def\valpha{\text{\scalebox{0.84}[1.02]{$\alpha$}}}   
\def\vepsilon{\upvarepsilon}

\newcommand{\BQ}{{\mathbb {Q}}} 
 
\newcommand{\BZ}{{\mathbb {Z}}}

\newcommand{\ra}{\rightarrow} 
\def\sumx{\sideset{}{^\star}\sum}

\def\nd{\mathrm{d}}

\newcommand{\delete}[1]{}

\theoremstyle{plain}

\newtheorem{thmnum}{Theorem}
 
\newtheorem{hypnum}{Theorem}
 
\newtheorem{lemnum}{Theorem}

\newtheorem{thm}{Theorem}   \newtheorem{hyp}[hypnum]{Hypothesis}  \newtheorem{lemma}[lemnum]{Lemma} 
\newtheorem{lem}{Lemma} 
\newtheorem*{lem*}{Lemma}

\theoremstyle{remark} 
\newtheorem{remark}{Remark} 
\newtheorem{defn}{Definition}

\numberwithin{equation}{section}

\begin{document}
\title[A Criterion for Equidistribution along the ${\Omega}$ Function]{A Criterion for Equidistribution along the $\boldsymbol{\Omega}$ Function over Polynomial Sequences with Applications}    


	\author[Z. Qi  and C. Zheng]{Zhi Qi and Cheng Zheng}
	\address{School of Mathematical Sciences\\ Zhejiang University\\Hangzhou, 310027\\China}
	\email{zhi.qi@zju.edu.cn}
	\address{School of Mathematical Sciences\\ Shanghai Jiao Tong University\\Shanghai, 200240\\China}
	\email{zheng.c@sjtu.edu.cn}
	
	\thanks{The first author was supported by National Key R\&D Program of China No. 2022YFA1005300.}


	\maketitle
\begin{abstract}
 Let $P (Y_1, ..., Y_d)$ be a certain fixed homogeneous polynomial of integral coefficients.  In this paper, we establish a quantitative  equidistribution criterion for the ergodic averages along $\Omega (P (n_1, ..., n_d))$.  
Consequently, by an estimate of Lachant, we prove the following variant of a theorem of Bergelson and Richter: if $P$ is an irreducible binary  cubic form and  $ (X, T)$ is a uniquely ergodic system with unique
invariant measure $\mu$, then for any $x \in X$ and $f \in C(X)$,
 \begin{equation*}
 	\lim_{N \rightarrow \infty}   
 	\frac 1 {N^2}  {\mathop{\sum\sum}_{n_1, n_2 \leqslant N}}    f \big( T^{ \Omega (|P (n_1, n_2)| ) } x \big) = \int_{X} f \   \mathrm{d} \mu . 
 \end{equation*} 
Moreover, we prove in the appendix a related conjecture of C\'espedes and Donoso over number fields. 
\end{abstract}

\section{Introduction}

For $n \in \mathbb{N} $, let  $\Omega (n)$ denote the total number of prime factors of $n$ counted with
multiplicity. Set $\Omega (0) = 0$. 

Let $ P   \in \BZ [\boldsymbol{Y}] $ be  a (non-constant) homogeneous polynomial of $d$ variables $\boldsymbol{Y} = (Y_1, ..., Y_d)$ and integral coefficients. Let $ \boldsymbol{n} = (n_1, \cdots, n_d)$ denote $d$-tuples in $\mathbb{N}^d$. 
 
\begin{hyp}  \label{hyp: criterion}
	 There exist $\valpha, \beta > 0$ so that for
	 \begin{equation}
	 \|	\theta\|\geqslant \frac 1 {(\log\log N)^{\valpha}} ,  \qquad (\|\theta \| = \min_{n \in \BZ} |\theta - n|), 
	 \end{equation}
	 we have uniformly 
	 \begin{equation}\label{eq:lachand} 
	 	\sumx_{   \max \{\boldsymbol{n} \}   \leqslant N   }  
	 	\mathrm{e}\!\left(\theta \, \Omega ( |P (\boldsymbol{n})|) \right) 
	 \Lt_{\valpha, \beta, P} 
	 	\frac{N^d}{(\log\log N)^{\beta}},\qquad \text{{\rm(}$\mathrm{e} (x) = \exp (2\pi i x)${\rm),}}
	 \end{equation}
as $N \ra \infty${\rm;} here and henceforth $\star$ means the primitive condition  $\mathrm{gcd}(\boldsymbol{n})=1$. 
\end{hyp}

\begin{hyp}\label{hyp: log log bound} 
	 \delete{Define \begin{equation}
	 	C^{\star}(N) =\text{\footnotesize  \bf \#} \big\{ \boldsymbol{n} \in \mathbb{N}^{s} : \max \{   \boldsymbol{n} \} \leqslant N , \, \mathrm{gcd} (\boldsymbol{n}) = 1 \big\}.
	 \end{equation}}
 As $N \ra \infty$, we have  
 \begin{equation} 
 	\sumx_{ \max \{   \boldsymbol{n} \} \leqslant N }
 	\Omega ( |P (\boldsymbol{n})|) 
 	\Lt_{P} N^d  \log\log N.
 \end{equation}
\end{hyp}

For a continuous map $T$ on a compact metric space $X$, if it admits exactly one $T$-invariant Borel probability measure $\mu$, then we say that the system $(X,T)$ or $(X,\mu,T)$ is uniquely ergodic.

\begin{thmnum}\label{thm: under hyp}
	 Assume that Hypotheses \ref{hyp: criterion} and \ref{hyp: log log bound} hold for some $ \valpha > 1/3 $, $\beta > 1/2$, and for a homogeneous polynomial $ P (\boldsymbol{Y}) \in \BZ [\boldsymbol{Y}]$.   Let $(X, \mu, T)$ be a uniquely ergodic system.  Then 
	 \begin{equation}\label{1eq: equi-dist}
	 		\lim_{N \rightarrow \infty} \frac 1 {N^d} \sum_{\max \{   \boldsymbol{n} \} \leqslant N }    f \big( T^{ \Omega (|P (\boldsymbol{n})|)}  x \big) = \int_{X} f \   \mathrm{d} \mu, 
	 \end{equation}
 for every $x \in X$ and $f \in C (X)$. 
\end{thmnum}

It will be seen shortly that Hypothesis \ref{hyp: criterion} is the main criterion for the equidistribution of the orbit $ \big\{T^{ \Omega (|P (\boldsymbol{n})|)} x : \boldsymbol{n} \in \mathbb{N}^d \big\}$.

The next lemma on the validity of Hypothesis \ref{hyp: criterion} is a direct consequence of \cite[Th\'eor\`eme 1.4]{Lachand-Thesis} (or in the special case $P (Y_1, Y_2) = Y_1^3 + 2 Y_2^3$ \cite[Th\'eor\`eme 1.3]{Lachand}). Lachand's theorems were built on the techniques in a series of profound works by Heath-Brown and Moroz  on the infinitude of primes represented by cubic binary polynomials in \cite{Heath-Brown-Cubic-1,Heath-Brown-Cubic-2,Heath-Brown-Cubic-3}. 

\begin{lem}[Lachand]\label{lem: Hyp A}
	Hypothesis \ref{hyp: criterion} holds for any $ \valpha < 1/2 $ and $\beta < 1$ if $P (Y_1, Y_2)$ is an irreducible   binary cubic form.  
\end{lem}

In \S \ref{sec: Hensel}, we shall show by the Hensel lemma that Hypothesis \ref{hyp: log log bound} is valid for a very large family of polynomials as in Definition \ref{defn: regular}.

\begin{defn}\label{defn: regular}
	Assume that the homogeneous polynomial $P (\boldsymbol{Y})$ has factorization of the form
	\begin{equation*}
		P (\boldsymbol{Y}) = c \prod_{j=1}^{r} P_j (\boldsymbol{Y})^{e_j},
	\end{equation*}
	where  $r, e_j \in \mathbb{N}$, $c \in \BZ \smallsetminus\{0\}$, and $ P_j  \in \BZ [\boldsymbol{Y}] $ are distinct primitive irreducible homogeneous polynomials. We say that $ P (\boldsymbol{Y})  $ is {\it factor-wise regular} if each projective hyper-surface $P_j (\boldsymbol{Y}) = 0$ is smooth over $\BQ$.  
\end{defn}

\begin{lem}\label{lem: Hyp B}
	  Hypothesis \ref{hyp: log log bound} holds if $P (\boldsymbol{Y})$ is factor-wise regular.    
\end{lem}

From Theorem \ref{thm: under hyp}, Lemmas \ref{lem: Hyp A} and \ref{lem: Hyp B}, we deduce the main theorem of this paper. 

\begin{thmnum}\label{thm: main, cubic}
	Let  $P $ be an  irreducible  cubic binary form of integral coefficients.  Let $(X, \mu, T)$ be uniquely ergodic.  Then 
	\begin{equation}\label{1eq: cubic}
		\lim_{N \rightarrow \infty} \frac 1 {N^2} 
		\mathop{\sum\sum}_{n_1, n_2 \leqslant N} f \big( T^{ \Omega (|P (n_1, n_2) | ) } x \big)  = \int_{X} f \   \mathrm{d} \mu, 
	\end{equation}
	for every $x \in X$ and $f \in C (X)$. 
\end{thmnum}

\begin{remark}
	Actually, in view of \cite[Th\'eor\`eme 1.4]{Lachand-Thesis},  Hypothesis \ref{hyp: criterion} holds for the cubic polynomial $P (a_1 + q Y_1, a_2 +  qY_2)$, with $ (a_1, a_2, q) = 1 $, so, with extra works, we may extend Theorem \ref{thm: main, cubic} in this setting. 
\end{remark}

In Appendix \ref{sec: num fields},  we shall  follow the same ideas to prove a conjecture of C\'espedes and Donoso over number fields \cite{CD-Num-Fields}, and it is turned into Theorem \ref{thm: main, quadratic} if the underlying field is imaginary quadratic.  For simplicity, let us restrict to those fields whose class number is equal to $1$. 

\begin{defn}\label{defn: norm form}
	Let $d = -1,  -2, -3, -7, -11, -19, -43, -67, -163.$ Define 
	\begin{equation*}
		\lambda = \left\{ \begin{split}
		&	\sqrt{d}, & & \text{ if } d \nequiv 1 \, (\mathrm{mod}\, 4), \\
		 & \displaystyle \frac {1 + \sqrt{d}} 2,  & & \text{ if } d \equiv 1 \, (\mathrm{mod}\, 4). 
		\end{split} \right.
	\end{equation*}
The quadratic binary form
\begin{equation}
	P (Y_1 , Y_2) = (Y_1 + \lambda Y_2) (Y_1 + \widebar{\lambda} Y_2) 
\end{equation}
is called the {\it norm form} of $\mathbb{Q} (\sqrt{d})$. 
\end{defn}

\begin{thmnum}\label{thm: main, quadratic}
	Let  $P $ be the  norm form of an imaginary quadratic field as in Definition \ref{defn: norm form}. Define 
	\begin{equation}
		C_P (N) = \text{\bf \small \#} \big\{ (n_1, n_2) \in \mathbb{N}^2 : P (n_1, n_2) \leqslant N \big\}. 
	\end{equation}  Let $(X, \mu, T)$ be uniquely ergodic.  Then 
	\begin{equation}\label{1eq: quad}
		\lim_{N \rightarrow \infty} \frac 1 {C_P (N)} 
		\mathop{\sum\sum}_{P (n_1, n_2) \leqslant N} f \big( T^{ \Omega (P (n_1, n_2)  ) } x \big)  = \int_{X} f \   \mathrm{d} \mu,  
	\end{equation}
	for every $x \in X$ and $f \in C (X)$. 
\end{thmnum}

\begin{remark}
	In Appendix \ref{sec: num fields}, if we  applied the Selberg--Delange method to Hecke $L$-functions for class-group characters instead of the Dedekind $\zeta$-function,   the restriction on the class number may be easily removed.  Further, we may extend Theorem \ref{thm: main, quadratic} for any irreducible quadratic binary form  of discriminant $\mathrm{disc}(F) < 0$. 
\end{remark}

\begin{remark}
	Let $\omega (n)$ denote the number of distinct prime factors of $n$.  As a rule of thumb, the results in this paper hold if $\Omega$ were replaced by $ \omega$. For example, the counterpart of Hypothesis \ref{hyp: log log bound} is indeed true for every polynomial $P (\boldsymbol{Y})$, as there is no need for the Hensel lemma. 
\end{remark}


Theorems   \ref{thm: main, cubic} and \ref{thm: main, quadratic} are  variants of the dynamic generalization of the Prime Number Theorem due to Bergelson and Richter \cite[Theorem A]{Bergelson-Richter}: 
\begin{equation}\label{1eq: BV PNT}
	\lim_{N \rightarrow \infty} \frac 1 {N } 
\sum_{n  \leqslant N} f \big( T^{ \Omega ( n ) } x \big)  = \int_{X} f \   \mathrm{d} \mu . 
\end{equation}


The case $P (Y_1, Y_2) = Y_1^2 + Y_2^2$ has been  proven by  Donoso et al. in \cite[Theorem D]{DLMS-Gaussian}; it is in the fashion of our Theorem \ref{thm: main, cubic} rather than Theorem \ref{thm: main, quadratic}.

Finally, we raise the question of whether a similar dynamic theorem is true  for   $$P (Y_1, Y_2) =  Y_1^2 + Y_2^4, \quad  Y_1^3 + Y_2^3; $$ the former  is motivated by the work of  Friedlander and Iwaniec \cite{FI-X2+Y4} while the latter  is a simple reducible cubic form (see \cite{Helfgott-Cubic-Red}).

\section{Proof of Theorem \ref{thm: under hyp}} \label{sec: proof, thm}
 
Define 
\begin{equation}
	 C^\star ({N}) = \text{\bf \small \#} \big\{ \boldsymbol{n} \in \mathbb{N}^d : \max \{ \boldsymbol{n}\} \leqslant N, \, \mathrm{gcd} (\boldsymbol{n}) = 1 \big\}. 
\end{equation}
For  $d > 1$, it is well known that 
\begin{align}\label{2eq: C(N) Nd}
	C^\star(N) \sim \frac {N^{d}} {\zeta (d)} . 
\end{align} 
Let  $(X, \mu, T)$ be uniquely ergodic. First, we would like to prove the primitive  analogue of \eqref{1eq: equi-dist}:
\begin{equation}\label{2eq: equi-dist}
	\lim_{N \rightarrow \infty} \frac 1 {C^{\star} (N)} \, \sumx_{\max \{   \boldsymbol{n} \} \leqslant N }    f \big( T^{ \Omega (|P (\boldsymbol{n})|)}  x \big) = \int_{X} f \   \mathrm{d} \mu,  
\end{equation}
for every $x \in X$ and $f \in C (X)$. Later in \S \ref{sec: primitive}, we shall show the deduction of  \eqref{1eq: equi-dist} from \eqref{2eq: equi-dist}. 

For $k \in \mathrm{\mathbb{Z}}$, we introduce   
\begin{equation}
	p_N^{\star} (k) = \frac { \text{\bf \small \#} \big\{ \boldsymbol{n} \in \mathbb{N}^d : \max \{ \boldsymbol{n}\} \leqslant N, \, \mathrm{gcd} (\boldsymbol{n}) = 1, \, \Omega (|P(\boldsymbol{n})| ) = k  \big\} } {\text{\bf \small \#} \big\{ \boldsymbol{n} \in \mathbb{N}^d : \max \{ \boldsymbol{n}\} \leqslant N, \, \mathrm{gcd} (\boldsymbol{n}) = 1 \big\}};
\end{equation} 
Given $N$, the sequence $ \{p_N^{\star}(k)\} $ is a probability measure on $\mathbb Z$ (of course, $ p_N^{\star}(k) = 0 $ if $k < 0$), and
\begin{align*}
	\frac1{C^{\star} (N) }\, \sumx_{\max{\{\boldsymbol{n}\}\leqslant N}}f \big(T^{\Omega(|P(\boldsymbol{n})|)}x \big)=\sum_{k}p_N^{\star}(k)f(T^kx).
\end{align*}
By the unique ergodicity of the map $T:X\to X$, it suffices to prove the following theorem, since,  in the case that $ \valpha > 1/3 $ and $\beta > 1/2$, it implies
\begin{equation}
\lim_{N \ra \infty} 	\sum_{ k  } \left| p_{N}^{\star} (k+1) - p_{N}^{\star} (k) \right| = 0. 
\end{equation}

\begin{thm}\label{thm: shift}
Assume Hypotheses  \ref{hyp: criterion} and \ref{hyp: log log bound}. 	As $N \ra \infty$, 
	\begin{equation}
		\sum_{ k   } \left| p_{N}^{\star} (k+1) - p_{N}^{\star} (k) \right| \Lt_{\valpha, \beta, P}  \frac 1 {\log_2^{ (1-\gamma)/ 3 } N}  ,
	\end{equation}
if we denote $ \log_2 N = \log \log N$ and  set $ \gamma = \min \{3\valpha, 2 \beta \}$. 
\end{thm}

\begin{remark}
	For an irreducible cubic form $P (Y_1, Y_2)$, Lachand's result (Lemma \ref{lem: A}) yields the quantitative bound: 
	\begin{equation*}
		\sum_{ k   } \left| p_{N}^{\star} (k+1) - p_{N}^{\star} (k) \right| \Lt_{\vepsilon, P}  \frac 1 {\log_2^{ 1/6-\vepsilon } N}  . 
	\end{equation*}
\end{remark}

\begin{proof}
	 For the sequence $\{p_N^{\star}(k)\} $, define its associated Fourier series by $$\phi_N(\theta)=\sum_{k}p_N^{\star}(k)\mathrm{e}(k\cdot\theta).$$ 
	 Then the Fourier series with coefficients $\big\{p_N^{\star}(k+1)-p_N^{\star}(k)\big\} $, which we denote by $\psi_N(\theta)$, may be written as 
	 \begin{align*}
	 	\psi_N(\theta)=\sum_{k}(p_N^{\star}(k+1)-p_N^{\star}(k))\mathrm{e}(k\cdot\theta)=(\mathrm{e}(-\theta)-1)\phi_N(\theta).
	 \end{align*}
	 By the Parseval formula, we have
	 \begin{align*}
	 	\sum_{k}|p_N^{\star}(k+1)-p_N^{\star}(k)|^2=\int_{-\frac12}^{\frac12}|\psi_N(\theta)|^2\nd \theta=\int_{-\frac12}^{\frac12}|e(-\theta)-1|^2|\phi_N(\theta)|^2\nd\theta.
	 \end{align*}
Now we truncate the integral at $ |\theta | = 1 / \log_2^{\valpha} N $. 
Trivially,  $ |\phi_N(\theta) | \leqslant 1 $ and $ |\mathrm{e}(-\theta)-1| \Lt \theta^2 $, so the contribution of the arc $|\theta | \leqslant 1 / \log_2^{\valpha} N$ is bounded by $ O \big(1 / \log_2^{3\valpha} N \big)$. On Hypothesis  \ref{hyp: criterion}, for $ 1 / \log_2^{ \valpha} N \leqslant |\theta| \leqslant 1/2 $, we have 
\begin{align*}
		|\phi_N(\theta)|\Lt_{\valpha, \beta, P}\frac1{\log_2^\beta N}, 
\end{align*}
and hence the integral on this arc is $ O \big( 1 / \log _2^{ 2\beta } N \big) $. We conclude with the following $l^2$-estimate for the sequence $\big\{p_N^{\star}(k+1)-p_N^{\star}(k) \big\} $, 
\begin{equation}\label{2eq: l2-estimate}
	 \sum_{k}|p_N^{\star}(k+1)-p_N^{\star}(k)|^2\Lt_{\valpha, \beta, P} \frac1{\log_2^{3\valpha} N} + \frac1{\log_2^{2\beta}   N} \Lt \frac 1 {\log_2^{\gamma} N},  
\end{equation}
where $\gamma = \min \{ 3 \valpha, 2 \beta \}$. 

Next, Hypothesis~\ref{hyp: log log bound} implies that
\begin{align*}
	\sum_{k}k\cdot p_N^{\star}(k)\Lt_{P}\log_2 N
\end{align*}
and hence for any $L>0$, we have
\begin{align}\label{2eq: sum of pk, k > L}
	\sum_{k\geqslant L} p_N^{\star}(k)\Lt_{P}\frac {\log_2 N} {L}.
\end{align}
Therefore, with the aid of the Cauchy and the triangle inequalities, we deduce from \eqref{2eq: l2-estimate} and \eqref{2eq: sum of pk, k > L} that 
\begin{align*}
	\sum_{k}|p_N^{\star}(k)-p_N^{\star}(k+1)|&=\sum_{k<L}|p_N^{\star}(k)-p_N^{\star}(k+1)|+\sum_{k\geq L}|p_N^{\star}(k)-p_N^{\star}(k+1)|\\
	&\Lt_{\valpha,\beta,P} \frac {L^{1/2} } { \log_2^{\gamma/2} N }  + \frac {\log_2 N} L.
\end{align*}
Now the proof is completed if we choose $L = \log_2^{(2+\gamma)/3} N$. 
\end{proof}

\subsection{Removal of the Coprimality Condition} \label{sec: primitive}
It is left to deduce \eqref{1eq: equi-dist} from \eqref{2eq: equi-dist}. To this end, we have
\begin{align}\label{2eq: split the sum}
	 \sum_{\max\{\boldsymbol{n}\}\leqslant N}f\big(T^{\Omega(|P(\boldsymbol{n})|)}x\big)= \sum_{m} \, \sumx_{\max\{\boldsymbol{n}\}\leqslant N/m}f\big(T^{\Omega(|P(\boldsymbol{n})|)} \big(T^{ \delta \, \Omega(m)}x \big) \big), 
\end{align}
where $\delta =\deg(P)$, and 
\begin{align}\label{2eq: split Nd}
	N^d = \sum_{m  } C^{\star} (N / m). 
\end{align}

Let $\vepsilon>0$. Fix $M > 0$ such that $$\sum_{m > M} \frac 1   {m^d} < \vepsilon. $$ By \eqref{2eq: equi-dist}, there exists $L>0$ such that for any $N>  L$, 
\begin{align*}
	\bigg|\frac1{C^{\star} ({N/m})} \, \sumx_{\max\{\boldsymbol{n}\}\leq N/m}f \big(T^{\Omega(|P(\boldsymbol{n})|)} \big(T^{\delta\, \Omega(m)}x \big) \big)-\int_X f \, \nd\mu \bigg|<\vepsilon , 
\end{align*}
for every  $m \leqslant  M $. By the triangle inequality, it follows from \eqref{2eq: split the sum} and \eqref{2eq: split Nd}  that 
\begin{align*}
	\bigg| \frac1{N^d}\sum_{\max\{\boldsymbol{n}\}\leq N}f(T^{\Omega(|P(\boldsymbol{n})|)}x)-\int_Xf \, \nd \mu   \bigg|
\end{align*}
is bounded by the sum of 
\begin{align*}
	\sum_{m \leqslant M}  \frac {C^{\star} (N/m)} {N^d} \bigg|  \frac1{C^{\star} ({N/m})} \, \sumx_{\max\{\boldsymbol{n}\}\leq N/m}f \big(T^{\Omega(|P(\boldsymbol{n})|)} \big(T^{\delta\, \Omega(m)}x \big) \big)-\int_X f \, \nd\mu   \bigg|, 
\end{align*}
and 
\begin{align*}
	\sum_{m > M}  \frac {C^{\star} (N/m)} {N^d} \bigg\{  \frac1{C^{\star} ({N/m})}   \,  \sumx_{\max\{\boldsymbol{n}\}\leq N/m} \big| f \big(T^{\Omega(|P(\boldsymbol{n})|)} \big(T^{\delta\, \Omega(m)}x \big) \big) \big|  +  \bigg|\int_X f \, \nd\mu   \bigg|  \bigg\}.
\end{align*}
For any $N > L$, it is clear that we may bound the former by $ \vepsilon$ and the latter by $O_{d, f} (\vepsilon)$ due to \eqref{2eq: C(N) Nd}.



\section{Proof of Lemma \ref{lem: Hyp B}} \label{sec: Hensel}

Since $\Omega (n)$ is additive, we may assume that $P (\boldsymbol{Y})$ itself is irreducible and the  the projective hyper-surface $P(\boldsymbol{Y})=0$ is smooth over $\mathbb Q$.  For $ P (\boldsymbol{n}) \neq 0 $, we have
\begin{equation*}
\Omega ( |P (\boldsymbol{n})| )	= \sum_{ p } v_{p} (P (\boldsymbol{n})) ,
\end{equation*}
where $v_p$ is the $p$-adic valuation. 
For $\max \{\boldsymbol{n}\} \leqslant N $, we have
\begin{equation*}
	\sum_{p \geqslant N}  v_{p} (P (\boldsymbol{n})) \leqslant    \frac {\log | P (\boldsymbol{n})| } {\log N} \Lt_{P} 1  ,
\end{equation*}
For $N \in \mathbb{N}$ and $\mu \in \mathbb{N}$, define 
\begin{equation}\label{3eq: A(N;pmu)}
	A_{P}^{\star} (N; p^{\mu}) = \text{\bf \small \#} \big\{ \boldsymbol{n} \in \mathbb{N}^d :  \max \{\boldsymbol{n}\} \leqslant N, \, \mathrm{gcd} (\boldsymbol{n}) = 1, \, P (\boldsymbol{n}) \equiv 0 \, (\mathrm{mod}\, p^{\mu})   \big\}. 
\end{equation}  
Then 
\begin{equation} \label{3eq: sum}
	\sumx_{\max \{\boldsymbol{n}\} \leqslant N } \Omega ( |P (\boldsymbol{n})| )	 = \sum_{p < N}   \sum_{\mu \Lt \frac {\log N}  {\log p} }  A_{P} (N; p^{\mu})  + O_P (N^d)  . 
\end{equation}

Since $P(\boldsymbol{Y})=0$ is smooth over $\mathbb Q$, for almost all primes $p$,   $ P (\boldsymbol{Y}) = 0 $ is smooth over $\mathbb{F}_p$. 

\begin{lem}\label{lem:hensel}
	Let $p$ be a prime such that $ P (\boldsymbol{Y}) = 0 $ is smooth over $\mathbb{F}_p$, then  
	\begin{equation}\label{3eq: A(N;p)}
		A_{P}^{\star} (N; p^{\mu})  \Lt_{P}  {N^d}/ {p^{\mu}}+N^{d-1} . 
	\end{equation} 
\end{lem}

\begin{proof}
	Let us assign each root \(\overline{\boldsymbol{y}}\in\mathbf{P}_{d-1} (\mathbb{F}_p) \) of the equation $ P (\boldsymbol{Y}) = 0 $ to an index \(j\) for
	which \(\partial P /\partial Y_j(\overline{\boldsymbol{y}})\ne0\).  Fix the other
	\(d-1\) integer coordinates (with at most $N^{d-1}$ many choices).  The resulting one-variable polynomial in \(Y_j\)
	has at most $\deg (P)$ many simple roots modulo \(p\).  For each such root, the Hensel  lemma
	gives a unique root class modulo \(p^{\mu}\). On the interval \([1,N]\), every
	residue class modulo \(p^{\mu}\) occurs at most \(N/p^{\mu}+1\) many times.  Thus the number
	of points as in \eqref{3eq: A(N;pmu)} assigned to \(j\) is   $ O_{P} (N^{d-1} ( {N} / {p^{\mu}}+1 )), $ and, by summing over $j = 1, ..., d$, we arrive at \eqref{3eq: A(N;p)}. 
\end{proof}

By Lemma \ref{lem:hensel}, we infer that the contribution from the good primes $p  $ to the sum in \eqref{3eq: sum} is bounded by 
\begin{align*}
	  N^{d} \sum_{p < N}    \bigg(\frac 1 {p-1} + \frac {\log N} {N \log p} \bigg) \Lt N^d \log \log N . 
\end{align*}

For those finitely many bad primes $p < N$  (i.e. $P(\boldsymbol{Y})=0$ is not smooth over $\mathbb F_p$), if we use the trivial bound $A_{P}^{\star} (N; p^{\mu}) < N^{d}$, their contribution would be $ O (N^d \log N) $. 
By the next lemma, however, we may readily improve this into $ O (N^d) $. 

\begin{lem}
	 There is a constant $c_p = c_p (P)$ for every prime $p$ such that  
	 \begin{equation}
	 	A_{P}^{\star} (N; p^{\mu})  \Lt_{P} p^{2c_p+1} N^{d-1} (N / p^{\mu} + 1),  
	 \end{equation}
 for any $\mu > 2  c_p$. 
\end{lem}

\begin{proof}
	  Let $c_p  $ be 
	  the smallest $v$ such that $ P (\boldsymbol{Y}) = 0 $ is smooth over $\BZ / p^{v+1} \BZ$. This lemma follows by the argument in the proof of Lemma \ref{lem:hensel}, using the generalized Hensel lemma as follows (see \cite[Theorem 7.3]{Eisenbud-CA}).  
	  \begin{lem*}[Hensel]
	  Let $f \in \BZ_p [Y]$. Let $x \in \BZ_p$ satisfy
	  \begin{align*}
	  v_p (	f (x)) > 2 v_p  ( f' (x) )  , 
	  \end{align*}
  then there is a unique root $y$ of $f (Y)$ such that
  \begin{align*}
  	v_p (x-y) >   v_p  ( f' (x) ). 
  \end{align*}
	  \end{lem*}
\end{proof}

\appendix 

\section{Proof of a Conjecture of C\'espedes and Donoso} 
\label{sec: num fields}

Recently, the number-field variant of \cite[Theorem A]{Bergelson-Richter} (see \eqref{1eq: BV PNT}) was established by C\'espedes and Donoso in  \cite{CD-Num-Fields}.  Let $\mathrm{K}$ be a number field, $ \SO_{\mathrm{K}}$ be its ring of integers, $G_{\mathrm{K}}$ denote the set of non-zero ideals in $\SO_{\mathrm{K}}$, and, for  $\mathfrak{n} \in G_{\mathrm{K}} $, let $\Omega_{\mathrm{K}} (\mathfrak{n})$ be the $\Omega$-function that counts the  number of prime ideals of $\mathfrak{n}$ with multiplicity. Let $\mathscr{N} (\mathfrak{n})$ denote the norm of $\mathfrak{n}$.  Define 
\begin{equation}\label{appeq: CK(N)}
	C_{\mathrm{K}} (N) = \text{\bf \small \#} \big\{ \mathfrak{n} \in G_{\mathrm{K}} : \mathscr{N} (\mathfrak{n})  \leqslant N \big\}. 
\end{equation} 
Note that (see \eqref{appeq: SD z=1}) 
\begin{align}\label{appeq: CK(N) - N}
	C_{\mathrm{K}} (N) \sim \rho_{\mathrm{K}} N, \qquad \rho_{\mathrm{K}} = \mathop{\mathrm{Res}}_{s=1} \zeta_{\mathrm{K}} (s),
\end{align}
where $\zeta_{\mathrm{K}} (s)$ is the Dedekind $\zeta$-function for $\mathrm{K}$. C\'espedes and Donoso proved in  \cite[Theorem 1.5]{CD-Num-Fields} that if  $(X, \mu, T)$ is uniquely ergodic, then for every $x \in X$ and $f \in C (X)$, 
\begin{equation}\label{appeq: PNT}
	\lim_{N \rightarrow \infty} \frac 1 { C_{\mathrm{K}} (N)  } 
	\sum_{\mathscr{N} (\mathfrak{n})  \leqslant N} f \big( T^{ \Omega_{\mathrm{K}} ( \mathfrak{n} ) } x \big)  = \int_{X} f \   \mathrm{d} \mu . 
\end{equation}
 It is proven via the approach suggested by Kanigowski and  Radziwi\l\l{} in \cite[Remark 1.3]{Bergelson-Richter} using the Sathe--Selberg theorem (the precise local law for the Erd\"os–Kac theorem) for $\Omega_{\mathrm{K}} ( \mathfrak{n} )$ in the work of Wu \cite{Wu-Selberg-Delange} (for the classical case of $\Omega (n)$, see  \cite{Erdos-K-P-Factors,Sathe,Selberg-Sathe}).   

Further, C\'espedes and Donoso conjecture that the limit on the left of \eqref{appeq: PNT} still exists if $ \Omega_{\mathrm{K}} ( \mathfrak{n} ) $ were replaced by $ \Omega (\mathscr{N} (\mathfrak{n})) $. The main theorem of this appendix manifests that their conjecture is true.  It should be stressed that the criterion (Lemma  \ref{lem: A}) that we use is weaker than the Sathe--Selberg theorem (but we believe that it is stronger than the Erd\"os–Kac theorem).

\begin{thmnum}
	\label{thm: num field}
	Let the notation be as above. Let  $(X, \mu, T)$ be uniquely ergodic. Then 
	\begin{equation}\label{appeq: PNT, 2}
		\lim_{N \rightarrow \infty} \frac 1 { C_{\mathrm{K}} (N)  } 
		\sum_{\mathscr{N} (\mathfrak{n})  \leqslant N} f \big( T^{ \Omega ( \mathscr{N} (\mathfrak{n}) ) } x \big)  = \int_{X} f \   \mathrm{d} \mu ,
	\end{equation}
	for every $x \in X$ and $f \in C (X)$.
\end{thmnum}

Theorem \ref{thm: main, quadratic} is a translation of Theorem \ref{thm: num field} in the special case that $ \mathrm{K} = \BQ (\sqrt{d}) $ is imaginary quadratic (so that the group of units is finite) and of class number $h_{\mathrm{K}} = 1$. 

\subsection{Proof of Theorem \ref{thm: num field}} 

First, in correspondence to Hypotheses \ref{hyp: criterion} and \ref{hyp: log log bound}, we have Lemmas \ref{lem: A} and \ref{lem: B} as follows. 

\begin{lemma}\label{lem: A}
	 We have 
	 \begin{equation}\label{appeq: bound exp sum}
	 	\sum_{\mathscr{N} (\mathfrak{n})  \leqslant N} \mathrm{e} (\theta \, \Omega (\mathscr{N} (\mathfrak{n})) ) \Lt_{\mathrm{K}} N  \exp \big\{ \! - 8 \|\theta\|^2 \log \log N \big\},
	 \end{equation}
 uniformly for all real $\theta$. 
\end{lemma}

\begin{lemma}\label{lem: B}
	 We have 
	 \begin{equation}\label{appeq: bound for Omega}
	 	 \sum_{\mathscr{N} (\mathfrak{n})  \leqslant N}   \Omega (\mathscr{N} (\mathfrak{n}))  \Lt_{\mathrm{K}} N  \log \log N . 
	 \end{equation}
\end{lemma}

It is clear that Lemma \ref{lem: A} is stronger than Hypothesis \ref{hyp: criterion} since \eqref{appeq: bound exp sum} implies that 
 for 
\begin{equation*}
	\|\theta\| \geqslant \frac 1 {(\log \log N)^{\valpha}},
\end{equation*}
with $   0 < \valpha < 1/2$,  we have uniformly
\begin{equation*}
	\sum_{\mathscr{N} (\mathfrak{n})  \leqslant N} \mathrm{e} (\theta \, \Omega (\mathscr{N} (\mathfrak{n})) ) \Lt_{\mathrm{K}, \valpha, \beta} \frac {N^d} {(\log \log N)^{\beta}},
\end{equation*}
for any $\beta > 0$.  At any rate, by the arguments in \S \ref{sec: proof, thm}, it is easy to deduce Theorem \ref{thm: num field} from Lemmas \ref{lem: A} and \ref{lem: B}.

\subsection{Sketch of Proof of Lemmas \ref{lem: A} and \ref{lem: B}} 
 
Let $\mathfrak{p} \in G_{\mathrm{K}}$ always denote prime ideals and $v_{\mathfrak{p}}$ the $ \mathfrak{p}$-adic valuation.  We have
\begin{equation}\label{appeq: Omega}
	\Omega_{\mathrm{K}} (\mathfrak{n}) = \sum_{\mathfrak{p}} v_{\mathfrak{p}} (\mathfrak{n}), \qquad  \Omega  (\mathscr{N} (\mathfrak{n})) = \sum_{\mathfrak{p}} f_{\mathfrak{p}} v_{\mathfrak{p}}  (\mathfrak{n}), 
\end{equation}
where $f_{\mathfrak{p}}$ is the residue degree of $\mathfrak{p}$ so that $ \mathscr{N} (\mathfrak{p}) = p^{f_{\mathfrak{p}}}$ if $ \mathfrak{p} \,|\, p $. 

For Lemma  \ref{lem: A}, the key observation is that the discrepancy at $f_{\mathfrak{p}} \geqslant 2$ is represented by an absolutely
convergent Euler product on $\mathrm{Re}(s) > 1/2$.   More explicitly, for $|z| \leqslant 1$ and $\mathrm{Re}(s) > 1$, define \begin{equation*}
 \SF  (s, z) =  \sum_{\mathfrak n } 
\frac{z^{\Omega (\mathscr{N}(\mathfrak n))}}{ \mathscr{N} (\mathfrak n)^s}, \qquad 	 \SG (s, z) =  \sum_{\mathfrak n } 
	 \frac{z^{\Omega_{\mathrm{K}} (\mathfrak n)}}{ \mathscr{N} (\mathfrak n)^s} 
	     ;
\end{equation*} 
the latter is exactly the Dirichlet series of Wu \cite{Wu-Selberg-Delange} in the special case that $ f (\mathfrak n)   \equiv 1 $ and $ F (\mathfrak n) =  \Omega_{\mathrm{K}} (\mathfrak n) $. Note that the results of Wu do not apply in our setting due to the requirement $ F (\mathfrak{p}) = 1 $.  It follows from \eqref{appeq: Omega} that 
\begin{equation*}
	 \SF  (s, z) = \mathscr{H}  (s, z) \cdot \SG  (s, z) ,
\end{equation*}
with
\begin{equation*}
	\mathscr{H}  (s, z) =  \prod_{ {\mathfrak p : f_{\mathfrak p} \geqslant 2}} 
	\frac{1-z \mathscr{N} ( \mathfrak p)^{-s}}
	{1-z^{f_{\mathfrak p}} \mathscr{N}( \mathfrak p)^{-s}}.
\end{equation*}
Note that  for  $ \mathfrak p \,  | \,  p $ with $ f_{\mathfrak p} \geqslant 2$, we have 
\begin{align*}
	\log \bigg(\frac{1-z \mathscr{N} ( \mathfrak p)^{-s}}
	{1-z^{f_{\mathfrak p}} \mathscr{N}( \mathfrak p)^{-s}}\bigg) = O \big( \big|\mathscr{N}( \mathfrak p)^{-s}\big| \big) = O \big(  \big|p^{-2s} \big| \big), 
\end{align*}
so $\mathscr{H}  (s, z) $ is holomorphic 
on $\mathrm{Re} (s) > 1/2$,  uniformly bounded on $\mathrm{Re} (s) \geqslant \delta > 1/2$.  By slightly modifying the argument of Wu, we may still prove the analogue of \cite[Theorem 1]{Wu-Selberg-Delange}:  for $ |z| \leqslant 1$, $z \neq 1$,  and $N$ large,   
\begin{equation}\label{appeq: SD |z|=1}
	 \sum_{ \mathscr{N}(\mathfrak n)\leqslant N} 
	z^{\Omega (\mathscr{N}(\mathfrak n))} = N \int_0^{\sigma_{\mathrm{K}}} 
	\SF_{0} (1-t,z) t^{-z}N^{-t}\,\mathrm dt
	+
	O_{\mathrm{K}}\big(N \exp \big\{\! -C_{\mathrm{K}}  \mathscr{N} (N) \big\} \big),
\end{equation}
and 
\begin{equation}\label{appeq: SD z=1}
	\sum_{ \mathscr{N}(\mathfrak n)\leqslant N} 
	1 = \rho_{\mathrm{K}} N
	+
	O_{\mathrm{K}}\big(N \exp \big\{\! -C_{\mathrm{K}}^{\, \prime}  \mathscr{N} (N) \big\} \big),
\end{equation}
where $   \sigma_{\mathrm{K}}  $ is a small constant, $\mathscr{N} (N) = (\log N )^{3/5} (\log \log N)^{-1/5} $,  and 
\begin{equation}
	 \SF_{0} (s,z) = \frac {\sin (\pi z)} {\pi s} \SF_{1} (s,z)   \exp  \{z \SQ   (s)   \},  
\end{equation} 
for $s$ near $1$, with  
\begin{equation}\label{appeq: F1(s,z)}
	\SF_{1} (s,z) = \zeta_{\mathrm{K}} (s)^{-z} \SF  (s,z), \qquad  \exp \{\SQ (s) \} = (s-1) \zeta_{\mathrm{K}} (s) . 
\end{equation} 
From \ref{appeq: SD |z|=1}--\eqref{appeq: F1(s,z)}, we deduce the Selberg--Delange formula for $ \Omega (\mathscr{N}(\mathfrak n))  $ (see  \cite{Delange,Selberg-Sathe} or \cite[\S 6]{Wu-Selberg-Delange} for the case of $\Omega (n)$ or $\Omega_{\mathrm{K}}(\mathfrak{n})$).    
\begin{lemma}
	 We have 
	 \begin{align} \label{appeq: asymp}
	 	\sum_{ \mathscr{N}(\mathfrak n)\leqslant N}  
	 	z^{\Omega (\mathscr{N}(\mathfrak n))}  = W (z) N (\log N )^{z-1}  + O_{\mathrm{K}} \big( N (\log N)^{\mathrm{Re}(z) - 2}   \big), 
	 \end{align}
uniformly for $|z| \leqslant 1$,  where 
 \begin{equation}
 	W (z) = \frac {\rho_{\mathrm{K}}^z \SF_{1} (1,z)} {\Gamma (z)} =  \frac{\rho_K^z}{\Gamma(z)}
 	\prod_{\mathfrak p}
 	 \bigg(1- \frac {z^{f_{\mathfrak p}} } {\mathscr{N} ( \mathfrak p)}\bigg)^{-1} {\bigg(1- \frac 1 {\mathscr{N}(\mathfrak p)} \bigg)^z}   . 
 \end{equation}
\end{lemma}

\begin{proof}
	 This lemma follows easily by  expanding $ \SF_{0} (1-t,z) $ into Taylor series at $ t = 0 $. Note that 
	 \begin{align*}
	 	\int_0^{\infty} t^{-z} N^{-t} \nd t = \Gamma (1-z) (\log N)^{z-1},  
	 \end{align*}
 while by the Euler reflection formula, 
 \begin{align*}
 	\frac {\sin (\pi z)} {\pi} \Gamma (1-z) = \frac 1 {\Gamma (z)}. 
 \end{align*}
\end{proof} 

As $W (z)$ is bounded on $|z| = 1$, we derive \eqref{appeq: bound exp sum} in Lemma \ref{lem: A} from \eqref{appeq: asymp} by setting $z = \mathrm{e}  (\theta) = \exp (2\pi i \theta)$. Now $$\big|(\log N)^{z - 1}\big| =  \exp \big\{
 (\cos(2\pi\theta) - 1)\log\log x
\big\} \leqslant \exp \big\{ \! - 8 \|\theta\|^2 \log\log x
\big\} . $$ 

Finally, let us prove Lemma \ref{lem: B}. It follow from \ref{appeq: CK(N)}, \ref{appeq: CK(N) - N}, and \ref{appeq: Omega} that 
\begin{align*}
	\sum_{\mathscr{N} (\mathfrak{n})  \leqslant N}   \Omega (\mathscr{N} (\mathfrak{n}))  =  \sum_{\mathscr N(\mathfrak p^{\mu} )  \leqslant N} f_{\mathfrak p}  C_{\mathrm K} \bigg(	\frac{N}{\mathscr N(\mathfrak p)^{\mu}}  \bigg) 
	\Lt_{\mathrm{K}}  N \sum_{ \mathscr N(\mathfrak p) \leqslant N }  \frac {1} { \mathscr N(\mathfrak p) -1 },
\end{align*}
where we have used the fact that $f_{\mathfrak{p}}$ is always bounded by the degree of $\mathrm{K}$. Now the Prime Ideal Theorem of Landau yields the bound $ O_{\mathrm{K}} (N \log \log N)$ as in \eqref{appeq: bound for Omega}.

\delete{For a rational prime $p$, we have
\begin{align*}
	p \SO_{\mathrm{K}} = \prod_{\mathfrak{p} \mid p}  \mathfrak{p}^{e_{\mathfrak{p}}}, \qquad \sum_{\mathfrak{p} \mid p} e_{\mathfrak{p}} f_{\mathfrak{p}} = \deg (\mathrm{K}).
\end{align*}
Thus the prime ideals $ \mathfrak{p}\, |\, p $ with $f_{\mathfrak{p}} = 1$ contribute at most  $\deg (\mathrm{K}) / p$, whereas those with $f_{\mathfrak{p}} \geqslant 2$ contribute at most $ \deg (\mathrm{K}) / p^2$. Consequently,
\begin{align*}
	\sum_{ \mathscr N(\mathfrak p) \leqslant N }  \frac {f_{\mathfrak{p}}} { \mathscr N(\mathfrak p) } \Lt_{\mathrm{K}} \sum_{ p \leqslant N } \frac 1 {p} + \sum_{p} \frac 1 {p^2} \Lt \log \log N,  
\end{align*}
as desired. }


\def\cprime{$'$}

\end{document}